\documentclass{amsart}
\newcommand{\Ni}{\mathbb{N}}

\newcommand{\KK}{\mathbb{K}}
\newcommand{\CC}{\mathbb{C}}

\newcommand{\NN}{\mathbb{N}}
\newcommand{\RR}{\mathbb{R}}
\newcommand{\QQ}{\mathbb{Q}}
\newcommand{\UU}{\mathbb{U}}
\newcommand{\ZZ}{\mathbb{Z}}
\newcommand{\notequiv}{\equiv \kern-0.95em / \mbox{ }}

\newtheorem{thm}{Theorem}[section]
\newtheorem{cor}[thm]{Corollary}
\newtheorem{lem}[thm]{Lemma}
\newtheorem{prop}[thm]{Proposition}
\newtheorem{rek}[thm]{Remark}

\newtheorem{defin}[thm]{Definition}
\newenvironment{deft}{\begin{defin} \rm}{\end{defin}}

\title{First order theory of cyclically ordered groups.}
\author{M. Giraudet, G. Leloup and  F. Lucas}
\date{November 3, 2013}
\begin{document}
\begin{abstract}
By a result known as Rieger's theorem (1956), there is a one-to-one correspondence, assigning 
to each cyclically ordered group $H$ a pair $(G,z)$ where $G$ is a totally ordered group and $z$ 
is an element in the center of $G$, generating a cofinal subgroup $\langle z\rangle$ of $G$, and such that the 
quotient group $G/\langle z\rangle$ is isomorphic to $H$. \\
\indent  We first establish that, in this correspondence, the first order theory of the cyclically 
ordered group $H$ is uniquely determined by the first order theory of the pair $(G,z)$.\\
\indent 	
Then we prove that the class of cyclically ordered groups is an elementary class and give an 
axiom system for it. \\
\indent	
Finally we show that, in opposition to the fact that all theories of totally Abelian 
ordered groups have 
the same universal part, there are uncountably many universal theories of Abelian 
cyclically ordered 
groups. We give for each of these universal theories an invariant, which is a pair of 
subgroups of the group of unimodular complex numbers.
\end{abstract}

\maketitle
\section{Introduction and basic facts.}\label{In}
\indent 
The study of cyclically ordered groups (c.o.g.) was initiated in \cite{R}.
Definitions and notations not given here, about c.o.g. and totally ordered groups (t.o.g.) can be 
found in \cite{F} (IV, 6, pp. 61-65), \cite{JP}, \cite{ZP} and \cite{Sw}. The terminology about 
model theory can be found in \cite{CK}. \\
\noindent We say that $(A,R)$ is a {\it cyclically ordered set} (or $R$ {\it is a cyclic order} on 
$A$) if $A$ is a set and $R$ is a ternary relation on $A$ satisfying the following axioms $R_1$ 
to $R_4$: \\
\indent 
$R_1:\; \forall x,y,z$ $ (R(x,y,z) \Rightarrow x \neq y \neq z \neq x) $, ($R$ is strict);\\
\indent 
$R_2: \; \forall x,y,z$ $ (x \neq y \neq z \neq x \Rightarrow (R(x,y,z) \mbox{ or } R(x,z,y)))$,  
($R$ is total);\\
\indent 
$R_3: \; \forall x,y,z$ $(R(x,y,z) \Rightarrow R(y,z,x))$,  ($R$ is cyclic);\\
\indent 
$R_4: \; \forall x,y,z,u$ $(R(x,y,z) \mbox{ and } R(y,u,z)\Rightarrow R(x,u,z))$,  
($R$ is transitive ).\\
\noindent We say that $(G,R)$ is a {\it cyclically ordered group (c.o.g.)} if $R$ is a cyclic order 
on the underlying set of $G$
which is compatible with the group law of $G$, i. e. satisfies:\\
\indent 
$R_5: \; \forall x,y,z,u,v$ $(R(x,y,z) \Rightarrow R(uxv,uyv,uzv))$.\\
\noindent It is easy to check that in a c.o.g.  with unit $e$, $R(e,x,y)$ implies $R(e,y^{-1},x^{-1})$. (Remark that $R$ is determined by its projection: 
$\{ (x,y);\; R(e,x,y) \}$). 
We shall often let $R(x,y,z,t,\dots)$ stand for $R(x,y,z)$ and $R(x,z,t)$ and $\dots$. \\
\noindent The language of c.o.g. will be here $L_c=\{\cdot,R,e,^{-1}\}$, where the first predicate 
stands for the group law, $R$ for the ternary relation, $e$ for the group identity and $^{-1}$ 
for the inverse function. (When considering Abelian c.o.g. we shall also use the usual symbols 
$+,0,-$). Remark that the theory of cyclically ordered groups has a finite set of universal 
axioms in $L_c$.\\
\noindent If $G$ is a c.o.g., $H$ is a normal subgroup of $G$, and $x\in G$, we shall let 
$\overline{x}$ 
stand for $xH$ whenever it yields no ambiguity.\\
\indent 
\noindent The {\it positive cone} of $(G,R)$ is the set $P(G)=P=\{x;\;
 R(e,x,x^2)\} 
\cup \{ e \}$. \cite{ZP}.\\
\indent 
\noindent Clearly $P\cap P^{-1}=\{e\}$ and $G=P\cup P^{-1}\cup \{x;\; x^2=e\}$. 
We shall set $\left| x \right|=x$ if $R(x^{-1},e,x)$, 
$\left| x \right|=x^{-1}$ if $R(x,e,x^{-1})$, and $\left| x \right|=x=x^{-1}$ if $x^2=e$. 
Remark that the positive cone $P$ of a c.o.g. does not always satify $P\cdot 
P \subseteq P$ (for 
example the (additive) c.o.g. $ \ZZ / 3\ZZ$ where $R(\overline{0},\overline{1},\overline{2})$ 
and $\ZZ $ is the 
additive group of integers, we have $\overline{1} \in P$ and $ \overline{1} + \overline{1} 
\notin P$).\\
\noindent If $G$ is a group, $Z(G)$ will denote its center. If $z \in G$ (or $X\subseteq G$) we say 
that $z$ (or $X$) is {\it central} if $z$ (or $X$) lies in the center of $G$: $z\in Z(G))$ (or 
$X\subseteq Z(G)$). If $(G,\leq )$ is a t.o.g. we say that $z$ (or $X$) is {\it cofinal in} $G$ if 
the subgroup generated by $z$: $\langle z\rangle$ (or $X$: $\langle X\rangle$) is cofinal in $(G,\leq)$.\\
\indent 
We must now give three fundamental constructions as follows in \ref{subsec11}, \ref{subsec12}, \ref{subsec13} below.
\subsection{Linear cyclically ordered groups.}\label{subsec11}
\indent 
A t.o.g. $(G,\leq)$ is cyclically ordered by the relation given by: $R(x,y,z)$ iff $(x<y<z 
\mbox{ or } y<z<x \mbox{ or } z<x<y)$. In this case we say that $(G,R)$ is the {\it cyclically 
ordered group associated to} $(G,\leq)$ and that $(G,R)$ is a {\it linear c.o.g.}. 
(Obviously a c.o.g. $(G,R)$ is linear if and only if $P\cdot P\subset P$). 
We have $e\leq x$ iff $R(e,x,x^2)$ iff $\left| x 
\right|=x$ (in this case $\left| x \right|$ has the same meaning in the linear c.o.g. it usually has 
in the t.o.g.).\\
\indent 
J. Jakub\'{\i}k and C. Pringerov\'a proved (\cite{JP} Lemma 3) that a c.o.g. $(G,R)$ is a linear c.o.g. 
iff it satisfies the following system of axioms: 
$\{\alpha\}\cup \{\beta _n \}_{n\in \Ni;\; n>1}$\\
\indent 
$\alpha$:  $\forall  x\neq e, \; x^2\neq e$\\
\indent 
$\beta _n$:  $\forall  x\neq e \; R(e,x,x^2)\Rightarrow R(e,x,x^n)$. 
\subsection{The winding construction.}\label{subsec12}(\cite{F})\\
\indent 
If $(G,\leq)$ is a t.o.g. and $z\in G$, $z>e$ is a central and cofinal element of $G$ then the quotient group $G/\langle z\rangle$ can be cyclically ordered by:\\
\indent 
$R(\overline{g},\overline{h},\overline{k})$ iff there are $g',h',k'$ such that\\
\indent 
$\overline{g}=\overline{g'}$, $\overline{h}=\overline{h'}$, 
$\overline{k}=\overline{k'}$ and $(e\leq g'<h'<k'<z \mbox{ or } e\leq h'<k'<g'<z \mbox{ or } 
e\leq k'<g'<h'<z)$.\\
\noindent $(G/\langle z\rangle,R)$ is the {\it wound-round c.o.g. associated to} $(G,\leq )$ and $z$.
This construction will be generalized later. The special cases below are of current use.
\subsubsection{Unimodular complex numbers.}
Let $\CC$ denote the field of complex numbers and $K=\{ x\in \CC;\; \left| x \right|=1\}=
\{ e^{i\theta};\; 0\leq \theta <2\pi \}$, equipped with the usual multiplication. For each 
$x \in \KK$ we  let $\theta (x)$ be the unique $\theta$ such that $x=e^{i \theta }$ and 
$0\leq \theta <2\pi$ and let $\KK$ be cyclically ordered by the relation:\\
\indent 
$R(e^{i\theta},e^{i\theta'},e^{i\theta''})$ iff ($\theta <\theta' < \theta''\mbox{ or }\theta' 
<\theta'' < \theta \mbox{ or }\theta ''<\theta < \theta'$) i.e.\\
\indent 
$R(x,y,z)$ iff ($\theta(x)< \theta(y) <\theta(z)$ or $\theta(y)< \theta(z) <\theta(x)$ or 
$\theta(z)< \theta(x) <\theta(y)$).\\
\noindent The c.o.g  $\KK$ is the wound-round c.o.g. associated to the additive ordered group of 
real numbers $(\RR,\leq )$ and $1$. It is of crucial importance.\\
\indent 
Let $\UU$ denote the torsion part of $\KK$ (the group of roots of $1$ in $\CC$). It is the 
wound-round associated to the additive ordered group $\QQ$ of rational numbers and $1$.
\subsubsection{Finite cyclic groups.}
\indent 
For each $n$ the finite cyclic group $\ZZ/n \ZZ$ is cyclically ordered by 
$R(\overline{\imath},\overline{\jmath},\overline{k})$ iff \\
\indent 
$(\exists i',j',k' \in \{0,1,\dots ,n-1\}, \overline{\imath}=\overline{\imath'}, 
\overline{\jmath}=\overline{\jmath'}, 
\overline{k}=\overline{k'} \mbox{ and } (i<j<k \mbox{ or } j<k<i \mbox{ or } k<i<j))$.\\
\noindent It is the wound-round c.o.g. associated to $(\ZZ, \leq )$ and $n$. 
Clearly $\UU$ embeds a copy of each such $\ZZ/n \ZZ$.\\
\indent 
In particular $\ZZ/2 \ZZ$  is cyclically ordered by the empty relation. A cyclically ordered group 
cannot have more than one element of order $2$. (Suppose $x\neq e, y\neq e, x\neq y, 
x^2=y^2=e$ and $R(e,x,y)$. Then $R(x,e,xy)$ and $R(y,xy,e)$ so $R(x,e,y)$, a contradiction). 
Hence if $(G,R)$ has an element of order $2$, then the c.o.g. $\ZZ/2 \ZZ$ is canonically embedded 
in $(G,R)$.
\subsection{Lexicographic product.}\label{subsec13}(\cite{Sw}, \cite{JP})\\
\indent 
If $(C,R)$ is a c.o.g. and $(L,\leq )$ is a t.o.g., we can define a {\it lexicographic cyclic order} on 
$C\times L$ by $R'((c,r),(c',r'),(c'',r''))$ iff\\
\noindent ($R(c,c',c'')$ or $(c=c' \neq c'' \mbox{ and } r<r')$ or $(c \neq c'=c'' \mbox{ and } 
r'<r'') $ or $(c=c'' \neq c' \mbox{ and } r''<r)$ or $(c=c'=c'' \mbox{ and } (r<r'<r'' \mbox{ or } 
r'<r''<r \mbox{ or } r''<r<r'))$).\\
\noindent We let $C\overrightarrow{\times} L$ denote this c.o.g. and call it the 
{\it lexicographic product} of $(C,r)$ and $(L,\leq )$.
The following classical results on c.o.g. are related to the above constructions.\\
\indent 
Rieger's theorem states that every cyclically ordered group $(G,R)$ can be obtained by the winding construction: there is a canonical t.o.g. $(uw(g), \leq )$ and a central and cofinal 
element $z_{G}$ in it, such that $(G,r) \cong ((uw(G), \leq )/\langle z_{G}\rangle)$.\\
\indent 
Swirczkowski's theorem states that every c.o.g. can be embedded in a lexicographic product 
$\KK\overrightarrow{\times} L$ for some t.o.g. $L$.\\
\indent 
The present model theoretic study of c.o.g. is based on those two theorems. In Section 
\ref{section2} we recall Rieger's theorem and we generalize the crucial construction. Section 
\ref{section3} deals with c-convex subgroups (an analogue of the convex subgroups of t.o.g.'s), 
normal c-convex subgroups, and problems about their characterization as kernels of 
c-homomorphisms, and finally we give a correspondence between proper c-convex subgroups of 
$(G,R)$ and proper convex subgroups of $(uw(G,R), \leq)$ in the nonlinear cases. In Section 
\ref{section4} we give a model theoretic version of Rieger's theorem: we prove (Theorem 
\ref{treq}) that two c.o.g. $(G,R)$ and $(G',R')$ have the same first order theory in $L_c$ iff 
so have $(uw(G),\leq ,z_G)$ and $(uw(G'),\leq ,z_{G'})$, in the language of o.g. with parameter. 
In Section \ref{section5} we recall Swirczkowski's theorem, and Zheleva's result which 
characterizes cyclically orderable groups in terms of their torsion parts. We prove (Theorem 
\ref{kkk}) that a group $G$ is cyclically orderable iff its center $Z(G)$ is cyclically orderable 
and the factor group $G/Z(G)$ is orderable (an analogue of the result of Kokorin and Kopitov 
characterizing totally orderable groups as those for which $Z(G)$ and $G/Z(G)$ are orderable). 
We derive from it (Theorem \ref{thm56}) an axiom system for the class of orderable c.o.g.. 
In Section \ref{section6} we study the universal theory of Abelian c.o.g.'s. Gurevich and 
Kokorin (\cite{GK}, \cite{Gp}) proved that two totally ordered Abelian groups satisfy the same 
universal formulas. For Abelian c.o.g., having or not an element of a given order, 
obviously gives different universal theories. We give (Theorem \ref{thm63}) a full classification 
of universal theories of Abelian c.o.g.'s. 
\section{Rieger's type constructions.}\label{section2}
\begin{thm}\label{th21} (Rieger, \cite{R}, \cite{F}). If (G,R) is a cyclically ordered group, there exists a linearly ordered group $(F,\leq)$ and a positive element $z\in F$ which is central and cofinal such that $(G,R)$ is the c.o.g. $(F,\leq)/\langle z\rangle$.
\end{thm}
The proof relies on the following theorem:
\begin{thm} (\cite{F}). If $(G,R)$
is a cyclically ordered group, then the structure $(uw(G),\cdot,\leq_R)$ is a linearly ordered group.\\
\indent Here: \\
1)	$(uw(G),\cdot)$ is the set $\ZZ\times G$, \\
2)	the order relation $\leq_R$ is defined by: $(m,g)\leq (m',g')$ iff $(m,g)=(m',g')$ or $m<m'$ or $m=m'$ and ($R(e,g,g')$ or $g=e$)),\\ 
3) the group law is given by: $(k,e)\cdot (m,h)=(k+m,h)$, $(k,g)\cdot (m,h)$ is either 
$(k+m,gh)$ if $ R(e,g,gh)$, or $ (k+m+1,gh)$ if $ R(e,gh,g)$, or $ (k+m+1,e) $ if $ gh=e\neq g $.
\end{thm}
\begin{rek}\label{rk23} One can also easily verify that:
\begin{enumerate}
\item If $g\neq e$ then $(k,g)^{-1}=(-k-1,g^{-1})$.
\item $(k,g)\cdot (k',h)=(k',h)\cdot (k,g)$ iff $gh=hg$.
\item $G$ is Abelian iff $uw(G)$ is Abelian.
\item The element $(1,e)$ which will be denoted by $z_G$, is central and cofinal in $uw(G)$.
\item $(uw(G),\leq_R)/\langle z_G\rangle \cong (G,R)$.
\item If $(F,\leq)$ is a linearly orderd group and $z$ is central and cofinal in $F$ and $R$ is the cyclic order defined on $uw(F/\langle z\rangle)$ in the winding construction, then 
$(uw(F/\langle z\rangle),\leq_R)$ is isomorphic to $(F,\leq)$.
\item If $(H,\leq)$ is a linearly ordered group, and $(G,R)$ is a c.o.g., we have $uw(G\overrightarrow{\times}H)=uw(G)\overrightarrow{\times}H$ where $uw(G\overrightarrow{\times}H)$ 
is the lexicographical product of linearly ordered groups and $G\overrightarrow{\times}H$ is the 
lexicographical product introduced in Section \ref{subsec13}.
\end{enumerate}
\end{rek}
\begin{deft} The linearly ordered group $(uw(G),\leq_R)$ will be named the {\it Rieger unwound} of $G$.
\end{deft}
\begin{lem}[winding, unwinding and substructures.]
\begin{enumerate}
\item If $(G,R)$ and $(G',R')$ are c.o.g., and $(G,R)$ is a substructure of $(G',R')$ then the unwound \\
$(uw(G,R),\leq_R)$ can be embedded in $(uw(G,R),\leq_R)$ with $z_G=z_{G'}$.
\item If $(L,\leq)$ and $(L',\leq)$ are linearly ordered groups and $(L,\leq)$ is a substructure of $(L', \leq)$ and for some $z\in L$ which is central and cofinal in $(L',\leq)$, $(L,\leq,z)$ is a substructure of $(L',\leq,z)$ then  z is also central and cofinal in $(L,\leq)$ and 
$(L/\langle z\rangle,R)$ is a substructure of $(L'/\langle z\rangle,R)$.
\end{enumerate}
\end{lem}
\begin{proof}. Easy to verify. \end{proof}
 We can generalize the winding construction:
 \begin{lem}\label{gwc}
 If $(G,\leq)$ is a linearly ordered group and $M$ a subgroup of $G$ which is discrete, central and cofinal, with first positive element $z$, and such that for each $g\in G$ there is $h$ in $M$ such that $h\leq g<hz$,( g lies between h and his successor is in $M$), then $G/M$ can be cyclically ordered following the winding construction. Moreover if $D$ is the convex hull of 
$\langle z\rangle$ in $G$, then $G/M=D/\langle z\rangle$.
 \end{lem}
\begin{proof}. For each $g$ there is a unique $g'$ such that $\overline{g}=\overline{g'}$ and $e<g'<z$, so we can define $R(\overline{g},\overline{h},\overline{k})$ iff there are $g',h',k'$ such that $\overline{g}=\overline{g'}$, $\overline{h}=\overline{h'}$, $\overline{k}=\overline{k'}$, and $(e\leq g'<h'<k'<z$ or $e\leq h'<k'<g'<z$ or $e\leq k'<g'<h'<z)$ and it is easy to verify that this is a cyclic ordering of $G/M$. \end{proof}
\indent Looking at the behavior of the center in the winding construction we have:
\begin{lem}\label{uwcenter} Let $(G,R)$ be a c.o.g. and $Z(G)$ the center of $G$. Then $Z(uw(G))=uw(Z(G))$ and $Z(uw(G))/\langle z_G\rangle \cong Z(G)$.
\end{lem}
\begin{proof}. Easy to verify ($z_G$ was introduced in \ref{th21}). \end{proof}
\section{c-convex subgroups.}\label{section3}
\indent 
It will be usefull to have precise settings for the notion of c-convex subgroup of a c.o.g.:
\begin{deft}(\cite{JP}) Each c.o.g. is a {\it c-convex subgroup} in itself. If $(G,R)$ is a c.o.g. and 
$H$ a proper subgroup of $G$ , it is a {\it c-convex subgroup} if it does not contain a non-unit element of order $2$ and it is a c-convex set ($\forall h\in H,\forall g\in G, (R(h^{-1},e,h) \mbox{ and } R(e,g,h)) \Rightarrow g\in H$).
\end{deft}
\indent 
It is clear that if a c.o.g. is linear, then a subgroup is c-convex iff it is convex in the associated t.o.g.. One can verify that a c-convex subgroup $K$ of a c-convex subgroup $H$ of a c.o.g. $(G,R)$ is a c-convex subgroup of $(G,R)$.
\begin{lem}If $H$ is a c-convex subgroup of $(G,R)$ and $x,x',y,y',z,z'\in G$ satisfy $R(x,y,z)$ and $\overline{x}\neq \overline{y} \neq \overline{z}\neq \overline{x}$ 
and $\overline{x}= \overline{x'}, \overline{y}= \overline{y'}, \overline{z}= \overline{z'}$ then $R(x',y',z')$ (where $\overline{t}$ denotes the class of $t$ in  $G/H$).
\end{lem}
\begin{proof}. Since $R$ is cyclic, it suffices to prove that $R(x,y',z)$. Suppose not, then $R(y',x,z)$, hence $R(x,y,z,y')$,  $R(e,zy^{-1},y'y^{-1})$. From $R(y',x,y)$ we have also $R(e,xy'^{-1},yy'^{-1})$. 
\indent 
If $R(yy'^{-1},e,y'y^{-1})$ since $H$ is c-convex $zy^{-1}\in H$ contradicting $\overline{z}\neq \overline{y}$. The case  when $R(y'y^{-1},e,yy'^{-1})$ is similar. \end{proof}
\begin{rek}
 Let $(G,R)$ be a c.o.g., $H$ be a proper normal subgroup of $G$, $S$ be the large cyclic order associated to $R$ ($S(x,y,z)$ iff $R(x,y,z)$ or $y=x$ or $y=z$), $\overline{S} $ the quotient relation ($\exists x',y',z'\in G, x x'^{-1}, yy'^{-1}, zz'^{-1}\in H $ and $S(x',y',z')$). 
Let $f$ be the associated surjective group homomorphism from $G$ onto $G/H$, and 
$\overline{R} $ the relation defined by $\overline{R}(\overline{x},\overline{y},\overline{z})$ iff $(\overline{x}\neq \overline{y}\neq \overline{z}\neq \overline{x}$ and $\overline{S}(\overline{x},\overline{y},\overline{z}))$.
Consider the three following conditions:
\begin{enumerate}
\item\label{cond1}
$H$ is a c-convex-subgroup of $G$
\item\label{cond2}
$\overline{R}$ is a cyclic order on $G/H$
\item\label{cond3}
$f$ is a c-homomorphism.
\end{enumerate}
\indent 
The conditions \ref{cond2} and \ref{cond3} are equivalent and they are consequences of \ref{cond1}. One can prove that if $G/H \not\subseteq \ZZ/2\ZZ$, then those three conditions are equivalent, but when 
$G=\ZZ/4\ZZ$, $H=2\ZZ$ and $f$ is defined by $f(x)=\overline{x}$ and $H=f^{-1}(e)$, then $f$ is a c-homomorphism but $H$ is not c-convex.
\end{rek}
\begin{thm} (\cite{JP}) Let $G$ be a nonlinear c.o.g., there is a largest c-convex subgroup of $(G,R)$, it will be denoted by $G_{0}$ and is called the subgroup of infinitely small elements of $G$. $G_{0}$ is linear and it is a normal subgroup of $G$.
\end {thm}
\indent 
Each proper c-convex subgroup of $(G,R)$ is linear and the set of c-convex subgroups of $(G,R)$ is linearly ordered by inclusion. It is closed by finite or infinite unions or intersections. For each $g\in G$ there is a smallest c-convex subgroup of $G$ containing $g$.\\
\indent 
One can prove that the map from $G$ into $uw(G)$ defined by $f(x)=(0,x)$ if $ x\in P$ and $f(x)=(-1,x)$ if not, when restricted to $G_{0}$ is a group homomorphism and an order isomorphism between $(G_{0},\leq_{R} )$ and $G_{0uw}=\{(0,x), x\in (P\bigcap G_{0})\} \bigcup \{(-1,x),x\neq e, x\in P^{-1}\bigcap G_{0})\}$.\\
\indent 
So, when $(G,R)$ is nonlinear we have an order isomorphism between the inclusion-ordered set of c-convex subgroups of $(G,R)$ and the inclusion-ordered set of convex subgroups of $(G,R)$.
\begin{deft}\cite{Sw} A c.o.g. $(G,R)$ is said to be {\it c-Archimedean} if for all $g$ and $h$ there is an integer $n$ such that $R(e,g^n,h)$ is not satisfied.
\end {deft}
\indent 
Examples.
\begin{enumerate}
\item $(G,R)$ is c-Archimedean iff it can be embedded in $\KK$, iff it is nonlinear and has no proper c-convex subgroup.
\item If $(G,R)$ is nonlinear and c-Archimedean then its unwound $(uw(G),\leq)$ is Archimedean.
\item \label{non-A} $\ZZ$ can be equipped with a non-Archimedean c.o.: consider $(G,R)=(\ZZ/3\ZZ)\overrightarrow{\times}\ZZ$, the subgroup $H$ generated in $G$ by $(\overline{1},1)$ equiped with the inherited cyclic order, and $C$ the c-convex subgroup of $(G,R)$ and $(H,R)$ generated by $(0,3)$. The group H is isomorphic to $\ZZ$ and the c.o.g. $(H,R)$ admits $C$ as a proper c-convex subgroup, so it is not c-Archimedean.
\end{enumerate}
\begin{deft}
 Let $(G,R)$ and $(G',R')$ be two c.o.g., $(G,R)$ being a substructure of $(G',R')$, $(G,R)$ is said to be {\it dense} in $(G',R')$ if\\
$\forall x',y' \in G' ((\exists z'\in G'\mbox{ } R(x',z',y')) \Longrightarrow (\exists z\in G \mbox{ }R(x',z,y'))$.
\end{deft}
\begin{rek}.\\
\begin{enumerate}
\item If $(G,R)$ and $(G',R')$ are linear then $(G,R)$ is dense in $(G',R')$ and $(G,\leq)$ is dense in $(G',\leq')$ for the associated linear orders.
\item
Each infinite subgroup of $\KK$ (equipped with the inherited cyclic order) is dense in $\KK$.
\end{enumerate}
\end{rek}
\indent 
As usual we say that a subgroup $H$ of $G$ is pure in $G$ if for each $h\in H$ and each integer $n$, if there exists $g\in G$ such that $g^n=h$ then there exists $h'\in H$ such that $h'^n=h$.
 \begin{rek}.\\
 \begin{enumerate}
 \item A c.o.g. can have torsion and the equation $x^n=y$ can have several solutions.
 \item A c-convex subgroup is not always a pure subgroup: let $(G,R)=(\ZZ/3\ZZ)\overrightarrow{\times}\ZZ$ and  $H$ the subgroup generated in $G$ by $(\overline{1},1)$ then $(0,3)=3(0,1)$, the subgroup $C$ generated by $(0,3)$ is a c-convex subgroup of $H$, $(0,3)$ is not divisible by $3$ in $C$, but it is divisible by $3$ in $H$. We shall see in \ref{pur} that if a c.o.g. contains a substructure isomorphic to $\UU$ then each of its c-convex subgroups is pure. 
\end{enumerate}
\end{rek}
\section{Elementary equivalence and substructure.}\label{section4}
\indent 
Now we prove a transfer principle between the elementary equivalence of two c.o.g. and the elementary equivalence of their Rieger-unwound totally ordered groups.
\begin{thm}\label{treq}
Let $(G,R)$ and $(G',R')$ be two c.o.g. then: \\
$(G,R)\equiv (G',R')$ iff $(uw(G),z_{G},\leq_{R})\equiv (uw(G'),z_{G'},\leq_{R'})$,\\
and the same property holds for elementary inclusion.
\end{thm}
\indent 
The proof goes through Lemmas \ref{Guw}, \ref{eqeluw} and \ref{uw<>}. In those lemmas $(G,R)$ and $(G',R')$ will be two c.o.g.. Let  $(uw(G),\leq_{R})$ be the unwound (linearly ordered group) of $(G,R)$ and $z_{G}\in uw(G)$ such that $(G,R)$ is canonically isomorphic to 
$uw(G)/\langle  z_{G} \rangle$. We consider the four different structures:
\begin{enumerate}
\item
$(G,R)$ in the language $L_{c}$ of c.o.g.
\item
$(uw(G),\langle z_{G}\rangle)$ in the language $L_{M}$ of pairs of ordered groups, containing the group law, the order relation and a predicate $M$ for a subgroup.
\item 
$(uw(G),z_{G})$ in the language $L_{z}$ of ordered groups with a specified element.
\item
$\ZZ\times G$ in the language $L_{ZG}$ with two unary predicates $Z$ and $G$ interpreted by $\ZZ \times \{0\}$ and $\{0\} \times G$, predicates for the group law and the order relation on $ Z$ and the cyclic order relation on $G$.
\end{enumerate}
\begin{lem}.\label{Guw}\\
\begin{enumerate}
\item
$(G,R)$ can be interpreted in  $(uw(G),\langle z_{G}\rangle)$ in the language $L_M$.
\item
$(uw(G),\langle z_{G}\rangle)$ can be interpreted in $\ZZ\times G$.
\end{enumerate}
\end{lem}
\begin{proof}.
\begin{enumerate}
\item
In $(uw(G),\langle z_{G}\rangle)$ with the language $L_{M}$ define \\
\indent 
$\Gamma =\{ g\in uw(G) ;\; \exists g'(g'=g \mbox{ or } g'=g^{-1})\mbox{ and } 
g'\geq e \mbox{ and } \forall t>e \mbox{ }(M(t)\Longrightarrow e\leq g'<t)\}$.\\
\indent 
For $g,h,k \in \Gamma$ define:\\
\hspace{1cm} $R(g,h,k) $ iff $(g<h<k \mbox{ or } h<k<g \mbox{ or }k<g<h )\mbox{ and } 
g\cdot h=k \mbox{ iff } M(ghk^{-1})$.\\
\indent 
Then $(\Gamma,R)$ is isomorphic to $(G,R)$.
\item
In $\ZZ\times G$ interpret $uw(G)$ by the whole set  and $\langle z_{G}\rangle$ 
(which is defined by the predicate $M$) by $\ZZ\times \{e\}$. The order relation $\leq $ will be defined from the order of $\ZZ$ and the cyclic order of $G$:
$v\leq v'$ iff\\
\hspace{1cm} $(\exists c,c',r,r')(Z(r),Z(r'),G(c), G(c') \mbox{ and }$\\
\hspace{1cm} $\mbox{  } v=r\cdot c \mbox{ and } v'=r'\cdot c' \mbox{ and } ((c=c' \mbox{ and } r=r') \mbox{ or }r<r' \mbox{ or } (r=r' \mbox{ and } R(e,c,c'))))$;\\
\indent 
and the group law is given by:
$v\cdot v'=v''$ iff\\
\hspace{1 cm}$(\exists c,c',c'',r,r',r'')(Z(r),Z(r'), Z(r''), G(c), G(c'), G(c'') \mbox{ and } v=c\cdot r \mbox{ and } v'=c'\cdot r' \mbox{ and }$\\
\hspace{1,5 cm}$v''=r''\cdot c'' \mbox{ and } c''=c\cdot c' \mbox{ and }((c=e \mbox{ or }c'=e \mbox{ or } R(e,c,c\cdot c')\Longrightarrow r''=r+r') \mbox{ and }$\\
\hspace{1,5 cm} \hspace{0,5 cm} $((c\neq e \mbox{ and }c'\neq e \mbox{ and }
R(e,c\cdot c',c))$ or\\
\hspace{1,5 cm} \hspace{1 cm} $(c\neq e \mbox{ and } c'\neq e \mbox{ and } c\cdot c'=e ))\Longrightarrow r''=r+r'+1 )).$
\end{enumerate}
\end{proof}
\begin{lem}\label{eqeluw}
 $(G,R)$ and $(G',R')$  are elementary equivalent in $L_{c}$ iff $(uw(G),\langle z_{G}\rangle)$ 
and $(uw(G'),\langle z_{G'}\rangle)$ are elementary equivalent in $L_{M}$, and the same property holds for elementary inclusion.
\end{lem}
\begin{proof}. We know that $(G,R)\equiv (G',R')$ implies (by model theoretic arguments  \cite{FV} ) $\ZZ\times G \equiv \ZZ\times G'$ which implies 
$(uw(G),\langle z_{G}\rangle,\leq)\equiv_{L_{M}} (uw(G'),\langle z_{G'}\rangle ,\leq ')$ 
by Lemma \ref{Guw},2. The converse follows from Lemma \ref{Guw},1, we obtain the result for elementary inclusion in the same way.
\end{proof}
\begin{lem}\label{uw<>}
$(uw(G),\langle z_{G}\rangle,\leq)\equiv_{L_{M}}(uw(G'),\langle z_{G'}\rangle,\leq')$ iff $(uw(G),z_{G})\equiv_{L_{z}}(uw(G'),z_{G'})$, and the same property holds for elementary inclusion.
\end{lem}
\begin{proof}.
\begin{enumerate}
\item
$z_{G}$ is definable in $\langle z_{G}\rangle$, it is the least positive element, hence: 
$$(uw(G),\langle z_{G}\rangle,\leq)\equiv_{L_{M}} (uw(G'),\langle z_{G'}\rangle ,\leq ') 
\mbox{ implies } (uw(G),z_{G})\equiv_{L_{z}}(uw(G'),z_{G'}),$$ (we do the same thing for elementary inclusion).
\item
We prove now the other part of the equivalence by proving that  $(uw(G),z_{G})\equiv_{L_{z}} (uw(G'),z_{G'} )$ implies $(G,R)\equiv (G',R')$ which is equivalent to the needed property.\\
\indent 
Suppose that $(uw(G),z_{G})\equiv_{L_{z}} (uw(G'),z_{G'} )$. There exist two $L_{z}$-isomorphic ultrapowers $(H,z)=(uw(G),z_{G})^{U}$ (respectively $(H',z')=(uw(G'),z_{G'})^{U'}$)
 of $(uw(G),z_{G})$ (resp. $(uw(G'),z_{G'}$ ). If $f$ is the isomorphism between these two structures, we have $f(z)=z'$, but the predicate $M$ is not in $L_{z}$ and may not be preserved by $f$.\\
\indent 
Define $D=\langle z_{G}\rangle^{U}\subseteq H $ 
(resp. $D'=\langle z_{G'}\rangle^{U'}\subseteq H '$), because of the definitions and properties of ultraproducts, $D$ (resp. $D'$) is a discrete subgroup of $H$ (resp. $H'$) with first element $z$ (resp. $z'$), and $D$ (resp. $D'$) is cofinal and central in $H$ (resp. $H'$).\\
\indent 
Let $B(z)$ (resp $B(z')$) be the minimal convex subgroup of $(H,\leq)$ containing $z$ (resp. of $(H',\leq ')$ containing $z'$).
The isomorphism $f$ between $(H,z,\leq)$ and  $(H',z',\leq)$, when restricted to $B(z)$ is an isomorphism between $(B(z),z,\leq)$ and $(B(z'),z',\leq)$. Now $z$ (resp. $z'$) is central and cofinal in $B(z)$ (resp. $B(z')$), and the corresponding c.o.g. are isomorphic: 
$(B(z)/\langle z\rangle, R)\approx (B(z')/\langle z'\rangle, R')$.\\
\indent 
Furthermore, $D$ and $D'$ are discrete and for each $g\in H$ (resp. $g'\in H'$) there is $h\in D$ 
(resp. $h'\in D'$) such that $g$ is between $h$ and its successor in $D$ ($g'$ is between $h'$ and its successor in $D'$). Hence 
by the generalized winding construction (\ref{gwc}) we can define a cyclic order on $H/D$ (resp. $H'/D'$) and we have $(H/D,R)\approx (B(z)/\langle z\rangle, R)$ and $(H'/D',R')\approx 
(B(z')/\langle z'\rangle, R')$. \\
\indent 
Hence we can conclude $(G,R)\equiv (G',R')$ because by (\ref{eqeluw}) $(G,R) \equiv (H/D,R)$ and $(G',R')\equiv (H'/D',R')$. \\
\indent 
All what we did can be done even when adding new symbols for all the elements in $uw(G)$ and we can prove that the elementary equivalence in this language: \\
\indent 
$(uw(G),z_G, \leq)\equiv_{L_{z}} (uw(G'),z_{G'}, \leq ')$ implies $(G,R)\equiv_{L(G)}(G',R')$ i.e. if $(uw(G),z_{G})$ is an elementary substructure of $(uw(G'),z_{G'})$ then $(G,R)$ is an elementary substructure of $(G',R')$.
\end{enumerate}
\end{proof}
This achieves the proof of Theorem \ref{treq}. 
\section{Embedding theorem of Swirczkowski and cyclic orderability.}\label{section5}
\begin{thm}\label{sw} \cite{Sw}
Let $(G,R)$ be a c.o.g., there are a linearly ordered group $(L,\leq)$ and an embedding $f$ of $(G,R)$ in the lexicographic product $\KK\overrightarrow{\times L}$. Such an embedding is called a representation of $(G,R)$.
\end{thm}
\indent 
Let $\pi_i$ ($i=1$ or $i=2$) be the $i^{th}$ projection associated with $f$, then:
\begin{enumerate}
\item 
$\pi_1\circ f$ does not depend on the representation \cite{JP}. Therefore we write $\pi_1(x)$ instead of $\pi_1(f(x))$ and $\pi_1$ instead of $\pi_1\circ f$. (When usefull we shall mention the domain: $\pi_{1,G}$). The image $\pi_1(f(G))$ will be called the {\it winding part} of $G$ and denoted by $K(G)$. We have $K(G_0) =1$.
\item
$\pi_2(f(G_0))=L\bigcap\pi_2(f(G))$.
\end{enumerate}
\indent 
As a consequence, if a c.o.g. $(G,R)$ is a substructure of $\KK$ and $f$ is a one-to-one c-isomorphism from $G$ to $\KK$, then for each $x$ we have $f(x)=x$.\\
\indent 
One can prove that if $g$ is a nontrivial c-homomorphism from $(G,R)$ into $(G',R')$ then $\ker(g)$ is c-convex iff $\pi_{1,G}=(\pi_{1,G'})\circ g$.\\
\indent 
As usual for a group $G$ its torsion part is $T(G)=\bigcup_{n\in \NN}\{g\in G/g^n=e\}$. Remark that for each prime $p$, $(G:pG)$ is $1$ or $p$, and that $T(G)\subseteq T(K(G))$.\\
\indent 
\begin{lem} \label{pur}
If a c.o.g. $(G,R)$ contains $\UU$ then each of its c-convex subgroups is pure.
\end{lem}
\begin{proof}. Let $f$ be an embedding of $(G,R)$ into $K(G)\overrightarrow{\times} L$ and $C$ be a proper c-convex subgroup of $(G,R)$, we have $f(C)\subseteq f(G_0)=L\bigcap f(G)$. Suppose $x \in C$, $y\in G$ and $y^n=x$. Let $f(x)=(1,t)$ and $f(y)=(\alpha,s)$. Then $f(y^n)=(\alpha^n ,s^n)$ so $\alpha^n=1$. We know that $G$ contains $\UU$ so there is $u\in G$ such that $f(u)=(\alpha,e)$. Therefore $f(yu^{-1})=(1,s)\in L\bigcap f(G)$. We have $y^n=x$ so $s^n=t$, $(1,s)^n=(1,s^n)\in f(C)$ which is linear and convex so $(1,s) \in f(C)$, i.e. $yu^{-1} \in C$ and $(yu^{-1})^n=x$.
\end{proof}
\indent 
Recall that a group is said to be {\it locally cyclic} if each finitely generated subgroup is cyclic. This is equivalent to being embedded in $(\QQ, +)$ or $\UU$ (see for instance \cite{Scn}).
\begin{thm} (Zheleva\label{zh} \cite{Zh}) A group $G$ is cyclically orderable iff its periodic part $T(G)$ is central and locally cyclic and $G/T(G)$ is orderable.
\end{thm}
\begin{proof}. We give a detailed proof.
\begin{enumerate}
\item If $G$ is cyclically ordered then $G$ can be embedded in $\KK\times L$ where $L$ is a linearly ordered group. Hence $T(G)$ is central. It is locally cyclic because it is a subgroup of $T(\KK)=\UU$. Now $G/(\UU\cap G)$ embeds in $(\KK\times L)/(\UU \times \{e\})=(\KK/\UU)\times L$ which is linearly orderable because it is abelian without nonzero periodic element.
\item
If $T(G$) is locally cyclic it can be embedded in $\UU$, so it is cyclically orderable. If $T(G)$ is central, then $G$ is a central extension of $T(G)$ by $G/T(G)$ and $G$ can be cyclically ordered by the lexicographic order defined in the following lemma.
\end{enumerate}
\end{proof}
\indent 
(Remark that, when applied to the Abelian case, this result is related to the result of G. Sabbagh \cite{Sa} giving the same characterization for the Abelian groups which can be embedded in the multiplicative group of a field).
\begin{lem}\label{ce}
If $C$ is an Abelian c.o.g., $L$ is a linearly ordered group and $G$ is a central extension of $C$ by $L$, then $G$ can be cyclically ordered by the following cyclic order $R$:\\
\indent 
Let $\{g_r, r \in L\}$ be a family of representatives of $L$ in $G$. Each element of $G$ is represented by a pair $(c,g_r)$ with $c\in C$ and $r\in L$. We define:\\
\indent 
$R((c,g_r),(c',g_{r'}),(c'',g_{r''}))$ iff $R(c,c',c'')$ or $(c=c' \neq c''$ and $\overline{g_r}<\overline{g_{r'}})$ or $(c\neq c'=c''$ and $\overline{g_{r'}}<\overline{g_{r''} )}$ or $(c=c''\neq c'$ and $\overline{g_r}>\overline{g_{r''}})$ or $(c=c'=c''$ and $(\overline{g_{r}} <\overline{g_{r'}}<\overline{g_{r''}}$ or $\overline{g_{r'}} <\overline{g_{r''}}<\overline{g_{r}}$ or $\overline{g_{r''}} <\overline{g_{r}}<\overline{g_{r'}})$.
\end{lem}
\begin{proof}.
We have to verify that this cyclic order is compatible with the group law in $G$. Let $(m_{r,s}, r,s\in L)$ be 
the factor system associated to the family of representatives $(g_r, r\in L)$ and if $R((c,g_r),(c',g_{r'}),(c'',g_{r''}))$ and $(c''',g_s)\in G$:\\
\indent 
$(c''',g_s)\cdot(c,g_r)=(c'''\cdot c ,g_{s\cdot r}\cdot m_{s,r})$\\
\indent 
$(c''',g_s)\cdot (c,g_{r'})=(c'''\cdot c ,g_{s\cdot r'}\cdot m_{s,r'})$\\
\indent 
$(c''',g_s)\cdot (c,g_{r''})=(c'''\cdot c ,g_{s\cdot r''}\cdot m_{s,r''})$.\\
\indent 
Remark that the factors $m_{u,v}$ are in $C$ and that in $G/C=L$ we have 
$\overline{g_u}\cdot \overline{g_v}=\overline{g_{u\cdot v}}$.\\
\indent 
If $(c=c'\neq c'' \mbox{ and } R(c,c',c''))$ then $(c'''\cdot c=c'''\cdot c'\neq c'''\cdot c'' 
\mbox{ and } R(c'''\cdot c,c'''\cdot c',c'''\cdot c''))$.\\
\indent 
If $(c=c'\neq c'' \mbox{ and } \overline{g_{r}}<\overline{g_{r'}})$ then $(c'''\cdot c=c'''\cdot 
c'\neq c'''\cdot c'' \mbox{ and } \overline{g_{s\cdot r}}<\overline{g_{s\cdot r'}})$.\\
\indent
If $(c=c''\neq c' \mbox{ and } \overline{g_{r'}}<\overline{g_{r''}})$ then $c'''\cdot c'=c'''\cdot 
c''\neq c'''\cdot c' \mbox{ and } \overline{g_{s\cdot r'}}<\overline{g_{s\cdot r''}})$.\\
\indent 
If $(c=c''\neq c' \mbox{ and } \overline{g_{r}}>\overline{g_{r''}})$ then $(c'''\cdot c=c'''\cdot 
c''\neq c'''\cdot c' \mbox{ and } \overline{g_{s\cdot r}}>\overline{g_{s\cdot r''}})$.\\
\indent 
If $(c=c'= c'' \mbox{ and } \overline{g_{r}}<\overline{g_{r'}}<\overline{g_{r''}})$ then  $\overline{g_{s\cdot r}}<\overline{g_{s\cdot r'}}<\overline{g_{s\cdot r''}}$.\\
\indent  
Hence $R((c''',g_s)\cdot (c,g_r),(c'''g_s)\cdot (c',g_{r'}),(c''',g_s)\cdot (c'',g_{r''}))$.\\
\indent 
In the same way we could obtain\\
\indent 
$R((c,g_r)\cdot (c''',g_s),(c',g_{r'})\cdot (c'''g_s),(c'',g_{r''})\cdot (c''',g_s))$.
\end{proof}
\indent 
Recall the following.
\begin{thm} (Kokorim and Kopitov \label{kk} \cite {KK}). A group $G$ is orderable iff its center $Z(G)$ and the factor group $G/Z(G)$ are orderable.
\end{thm}
\indent 
We can obtain a similar result for cyclically ordered groups.
\begin{thm}\label{kkk} 
A group $G$ is cyclically orderable iff its center $Z(G)$ is cyclically orderable and the factor group $G/Z(G)$ is orderable.
\end{thm}
\begin{proof}.
\begin{enumerate}
\item
If $Z(G)$ is cyclically orderable and $G/Z(G)$ is orderable, then $G$ is a central extension of $G/Z(G)$ by $Z(G)$ and then can be cyclically ordered using Lemma \ref{ce}.
\item
Let $G$ be cyclically ordered, $uw(G)$ is totally ordered and by Theorem \ref{kk} $uw(G)/Z(uw(G))$ is orderable. The subgroup $\langle z_{G}\rangle$ is normal in uw(G) (and contained in $Z(uw(G))$), hence 
$uw(G)/Z(uw(G))\approx (uw(G)/\langle z_{G}\rangle)/(Z(uw(G))/\langle z_{G}\rangle)$. 
We have by Lemma \ref{uwcenter} $Z(uw(G))/\langle z_{G}\rangle\approx Z(G)$ hence $uw(G)/Z(uw(G))\approx G/Z(G)$ and $G/Z(G)$ is orderable.
\end{enumerate}
\end{proof}
\indent 
Finally we give, in the language of groups, a system of axioms for cyclic orderability. First recall the characterization of orderable groups given by Onishi and Los.
\begin{thm} (Onishi, Los \cite{O}, \cite{L}, see also \cite {KK} ch. 2 th. 3)
A group $G$ is orderable iff for any finite set of non identity elements $x_1, \dots ,x_n$ there is an $\epsilon=(\epsilon_1,\dots, \epsilon_n) \in \{1,-1\}^n$ such that $e$ does not belong to the semigroup generated by the conjugates of $x_1^{\epsilon _1},\dots ,x_n^{\epsilon_n}$. Furthermore the class of orderable groups is axiomatizable by the following family $(O_n \mbox{ } n\in \NN)$ of formulas where $O_n$ is:
$\forall x_1,\dots ,x_n((\bigwedge \bigwedge_{i\in n} x_i \neq e )$ \\
\indent 
$ \Rightarrow \bigvee \bigvee_{(\epsilon_1,\dots ,\epsilon_n)\in \{1,-1\}^n} \bigwedge \bigwedge_{k\in n,i_1,\dots ,i_k \in n} \forall y_{i_1}, \dots,y_{i_n}\; e\neq  \Pi_{j=1, \dots,k} (y_{i_j})^{1}\cdot x_{i_j}^{\epsilon_{i_j} } y_{i_j})$.
\end{thm}
\begin{thm}\label{thm56} Let $G$ be a group, the following are equivalent:
\begin{enumerate}
\item
$G$ is cyclically orderable
\item
$T(G)$ is locally cyclic and $G/Z(G)$is orderable
\item
$G$ satisfies the following system of axioms $a_n, b_n, n\in \NN$
\begin{enumerate}
\item
$a_n: \forall x_0,\dots ,x_n((\bigwedge \bigwedge_{0<i< n} x_i^n = e \Rightarrow \bigvee \bigvee_{0<i<j<n}x_i=x_j)$ \\
\indent 
(i.e. the n-torsion part of $G$ has dimension at most one).
\item
$b_n: \forall x_1,\dots ,x_n(\bigwedge \bigwedge_{i\in n}x_i\notin Z(G))\Rightarrow \bigvee \bigvee_{(\epsilon_1,\dots ,\epsilon_n)\in \{1,-1\}^n}$$\\
$$\bigwedge \bigwedge_{k\in n,i_1,\dots ,i_k \in n} \forall y_{i_1}, \dots,y_{i_n}(\Pi_{j=1, \dots ,k}(y_{i_j})^{-1}\cdot 
x_{i_j}^{\epsilon_{i_j}}\cdot y_{i_j}) \notin Z(G)$.
\end{enumerate}
\end{enumerate}
\end{thm}
\begin{proof}. Using Theorem \ref{kk} $G$ is cyclically orderable iff  $G/Z(G)$ is orderable and $Z(G)$ is cyclically orderable. Now  by Theorem \ref{zh} $Z(G)$ is cyclically orderable iff  $T(Z(G))$ is is central in $Z(G)$ and locally cyclic and $Z(G)/T(Z(G))$ is orderable. Remark that:\begin{enumerate}
\item
If $G/Z(G)$ is orderable then $T(G)\subseteq Z(G)$ so $T(G)=T(Z(G))$.
\item
$Z(G)/T(G)$ is orderable because it is Abelian without torsion part.
\end{enumerate}
Therefore $G$ is cyclically orderable iff $T(G)$ is locally cyclic and $G/Z(G)$ is orderable.
\end{proof}
\section{Remarks on the universal theory of Abelian cyclically ordered groups.} \label{section6}
\indent (Looking at Abelian c.o.g., the group law will be denoted additively).\\
\indent 
 Remark first that, without the Abelian hypothesis, one can prove that two c.o.g. have the same universal theory iff their unwound linearly ordered groups have the same universal theory in the language with a constant realised by $z_G$: $(G,R)\equiv_{\forall }(G',R')$ iff $(uw(G),z_G)\equiv_{\forall }(uw(G'),z_{G'})$. The proof uses the same argument that the one given for Theorem \ref{treq}. \\
\indent 
Gurevich and Kokorin (\cite{GK} in Russian, see also \cite{Gp}) proved that two linearly ordered Abelian groups satisfy the same universal formulas. For Abelian c.o.g. the existence, or not, of an element of a given torsion type gives different universal theories.
 C.o.g. without torsion part can also have different universal theories, for example the formula $\exists x R(x,2x,3x,4x,0)\wedge \neg R(x,2x,3x,4x,5x,0)$ is satisfied in the c.o.subgroup of $\KK$ generated by $e^{2i\pi \vartheta}$ with $\vartheta $ an irrational number, but it is not satisfied in the c.o.g. $H$ of Example \ref{non-A} in Section \ref{section3}. \\
\indent 
Recall that if $G$ is an Abelian linearly ordered group and $C$ is a convex subgroup of $G$ then $G$ is elementary equivalent to the lexicographical product $G/C\overrightarrow{\times} C$. Here we have an analogue:
\begin{thm} \label{lexpur}
If $(G,R)$ is an Abelian c.o.g. and $C$ a c-convex subgroup of $G$ which is pure in $G$,
then $(G,R)\equiv (G/C)\overrightarrow{\times}C$.
\end{thm}
\begin{proof}. Let $ L_2$ be the language obtained by adding to the language of c.o.g. a
predicate P which we interpret by $C$. Let $(G_1, C_1)$ be an $\omega _1$-saturated model
of the theory of $(G,C)$, then by a result of Eklof and Fisher \cite{EF}, when
considering the group structure we have $G_1\simeq H\times C_1$, for some subgroup $H$ of 
$G_1$. $C$ is linear and c-convex in
$G$ so $C_1$ is linear and c-convex in $G_1$. 
Now we verify that the cyclic order of $G_1$ coincides with the lexicographic cyclic 
order $R_{\times}$ on $H \overrightarrow{\times} C_1$ defined by: for $a$, $a'$ in $H$ and $b$, $b'$ 
in $C_1$, 
$R_{\times}(0,a+b,a'+b')$ iff
($R(0,a,a')$ or ($a=0\neq a'$ and $0<b$)
or ($0\neq a=a'$ and $b<b'$)
or ($a\neq 0=a'$ and $b'<0$)
or ($a=a'=0$ and $R(0,b,b')$)). It is clear that $R_{\times}$ and $R$ coincide on $H$ and on
$C_1$. Moreover, because $C_1$ is c-convex, we have when $a\neq 0$ and $a'\neq 0$: 
$R(0,b,a'+b')$ iff $b>0$, $ R(0,a+b,b')$ iff $b<0$, and $R(0,a+b,a+b')$ iff $
R(0,b'-b,-a-b)$ iff $b'-b>0$. We have $G_1/C_1\equiv G/C$, $C_1\equiv C$.
We saw (Remark \ref{rk23}) that for each  $G$ and $H$
$uw(G \overrightarrow{\times} H)=uw(G) \overrightarrow{\times} H$.
We also have (\cite{FV}) that for t.o.g. the lexicographical product preserves the
elementary equivalence, so passing through the unwounds we obtain:
$(uw(G),z_G)\equiv (uw(G_1),z_{G_1})\equiv (uw(G_1/C_1)\overrightarrow{\times}  
C_1,(z_{G_1}+C_1,0)) \equiv
(uw(G/C)\overrightarrow{\times} C,(z_G+C,0)) \equiv (uw(G/C\overrightarrow{\times} C),(z_G+C,0))$,
Hence $G\equiv G/C\overrightarrow{\times} C$.
\end{proof}
\begin{thm}\label{thm63} 
Let $G$ and $G'$ be two Abelian c.o.g.. Then $G$ and $G'$ satisfy the same universal formulas if and only if their torsion subgroups are isomorphic and, 
either $K(G)$ and $K(G')$ are finite and isomorphic,
or $K(G)$ and $K(G')$ are infinite. 
\end{thm} 
\indent For $x$ in  the linear part $G_0$ of $G$, we will denote by $H_2(x)$ the convex subgroup of $G_0$ generated by $x$, and by $H_1(x)$ the greatest convex subgroup of $G_0$ which doesn't contain $x$. We know that the quotient group $H_2(x)/H_1(x)$ is Archimedean, it will be called the {\it Archimedean component of} $G_0$ {\it associated to} $x$. \\[2mm] 
\indent 
One easily checks that the cardinal of  $K(G)$ is $n$ if and only if 
$G$ satisfies the formulas \\ 
$\exists x,\; R(0,x,\dots, nx)\wedge 
\neg R(0,x, \dots,nx,(n+1)x)$, and, for every $m>n$, \\ 
$\forall x,\; 
\neg R(0,x,\dots,mx) \vee R(0,x,\dots,mx,(m+1)x)$, and that the torsion subgroup of $G$ is 
determined by the formulas $\exists x,\; x\neq 0 \wedge nx=0$ and $\forall x, \;x\neq 0 
\Rightarrow nx \neq 0$. It follows that if two Abelian c.o.g.  $G$ and $G'$ satisfy the same universal formulas then their torsion subgroups are isomorphic, and $K(G)$ and $K(G')$ are either finite and isomorphic or both infinite. For proving the converse, we can assume that $G'$ is $\omega_1$-saturated, and by a property of universal formulas, it is sufficient to show that for every finite subset $E$ of $G$ there exists a finite subset $E'$ of $G'$ and a one-to-one mapping $x\mapsto x'$ from $E$ onto $E'$ such that for every  $x,y,z$ in $E$, we have: $R(x,y,z) \Leftrightarrow R(x',y',z')$ and 
$x=y+z\Leftrightarrow x'=y'+z'$. So Theorem \ref{thm63} is a consequence of 
Proposition \ref{prop2} a) and Corollary \ref{cor5} above.  
\begin{prop}\label{prop2} Let $G$ be an $\omega_1$-saturated infinite Abelian c.o.g.. \\ 
a) If $K(G)$ is infinite, then $K(G)\simeq \KK$. \\ 
b) The linear part $G_0$ of $G$ has no maximal proper convex subgroup. \\ 
c) For every $x$ in the positive cone of $G_0$, there exists 
$y\in G_0$ such that $R(0,x,y)$ and the Archimedean component of $G_0$ associated to $y$ is isomorphic to $\RR$. 
\end{prop} 
\begin{proof}.\\ 
\indent a) If $n\geq 2$ and $\overline{x}=e^{i\alpha}$, where $\alpha \in [0,2\pi[$, 
then the formula $R(0,x,2x,\dots,(n-1)x)\; \& \; \neg 
R(0,x,2x,\dots,(n-1)x,nx)$ implies $\frac{2\pi}{n}\leq \alpha \leq \frac{2\pi}{n-1}$. 
Let $m\in \NN^*$, if for every $y$ such that $\overline{y}=e^{i\beta}$, where 
$\frac{2\pi}{n}\leq \beta \leq \frac{2\pi}{n-1}$, 
 we have $R(0,my,x)$, then we have $2\pi\frac{m}{n} \leq \alpha$. Now, if there  
exists $y$ such that $\overline{y}=e^{i\beta}$, where 
$\frac{2\pi}{n}\leq \beta \leq \frac{2\pi}{n-1}$, and 
$R(0,x,(m+1)y)$, then: $\alpha \leq 2\pi\frac{m+1}{n}$. We fix some $\alpha \in [0,2\pi[$, 
and for every $n\geq 2$, denote by $m_n$ the integer such that 
$2\pi\frac{m_n}{n} \leq \alpha< 2\pi\frac{m_n+1}{n}$. If $x$ is a realization of the type: 
$$\{[\forall y,(R(0,y,2y,\dots,(n-1)y)\& \neg R(0,y,2y,\dots,(n-1)y,ny)) \Rightarrow 
R(0,m_ny,x)]\&$$ 
$$[\exists y,R(0,y,2y,\dots,(n-1)y)\&\neg R(0,y,2y,\dots,(n-1)y,ny) 
\&R(0,x,(m_n+1)y)];\; n\geq 2\}.$$ 
then $\overline{x}=e^{i\alpha}$. Since $K(G)$ is infinite, it is dense in $\KK$, 
and every finite subset of this countable type has a realization. 
By $\omega_1$-saturation the type has a realization in $G$, 
which proves that $K(G)=\KK$. \\ 
\indent\indent b) First we show that $G_0$  is nontrivial. This is obvious if
$K(G)$ is finite, since $G$ is infinite. Otherwise,  $K(G)$ is dense in $\RR$, hence for every  
$n\geq 2$ there exists $x\in G$ such that $R(0,x,nx)$. 
So, every finite subset of the type $\{ R(0,x,nx);\; 
n\geq 2\}$, which 
cha\-rac\-t\-e\-ri\-zes the elements that belong to the positive cone of $G_0$, 
has a realization in $G$. By $\omega_1$-saturation the whole type has a realization 
in $G$, which proves that $G_0$ is nontrivial. 
Let $x$ belong to the positive cone of $G_0$ and consider the type 
$\{R(0,x,nx,y,ny);\; n\geq 2\}$, 
which says, on the one hand that $y$ is an element of $G_0$, and on the other hand that $y$ is greater than $\NN x$ within $G_0$. A finite subset of this type is a consequence of some formula $R(0,x,nx,y,ny)$, and it has a realization $(n+1)x$, hence by
$\omega_1$-saturation the whole type has a realization in $G$, and $x$ is not cofinal 
in $G_0$. It follows that $G_0$ has no greatest proper convex subgroup. \\ 
\indent c) Let $x$ belong to the positive cone of $G_0$, $2< \xi$ be an  irrational element of 
$\RR_+$, and $(a_n/b_n)$ be a strictly increasing sequence of rational numbers such that for every 
$n\in \NN$ $a_n/b_n < \xi < (a_n+1)/b_n$. Consider 
the type $\{R(0,x,y,ny), \;R(0,x,z,nz), \;R(a_ny,b_nz,(a_n+1)y);\; n\geq 2\}$, 
which says that there exist $y$ and $z$ in $G_0$ which are greater than 
$x$ and such that the class of $z$ in $H_2(y)/H_1(y)$ is the irrational number $\xi$ 
(where the class of $y$ is assumed to be $1$); hence, $H_2(y)/H_1(y)$ is dense in $\RR$.  
A finite subset of this type is generated by some formula 
$R(0,x,y,n_0y)\&R(0,x,z,n_0z)\&R(a_{n_1}y,b_{n_1} z,(a_{n_1}+1)y)\&
R(a_{n_2}y,b_{n_2} z,(a_{n_2}+1)y)$, and it has a realization: $y=x$ and $z=
(3a_{n_1}b_{n_2}+2)x$. By $\omega_1$-saturation, the type has a  
realization.  Now, let $\xi'$ in $\RR$ and let $(a_n'/b_n')$ be a strictly increasing sequence of rational numbers such that, for every $n\in \NN$, $(a_n'/b_n') \leq 
\xi'<((a_n'+1)/b_n')$. Consider the type $\{R(0,a_n'y,b_n' t,(a_n'+1)y), \;R(0,t,nt)
;\; n\geq 2\}$, which says that the class of $t$ is $\xi'$. A finite subset of the type that we defined is characterized by some formula 
$R(0,a_n'y,b_n' t,(a_n'+1)y)\& R(0,t,nt)$ which has a realization since $H_2(y)/H_1(y)$ 
is dense in $\RR$. By $\omega_1$-saturation, $\xi'$ has an  
image in $\RR$ under the canonical embedding which associate $1$ to $y$. This proves that this embedding is an isomorphism. 
\end{proof}
\begin{prop}\label{prop3} 
Let $G$ be a finitely generated Abelian c.o.g., $T(G)$ be its torsion subgroup 
and $T(K(G))$ be the torsion subgroup of 
$K(G)$. Then $G$ decomposes as a disjoint union 
$G=C \cup (x_0 +C)\cup \cdots \cup ((n-1)x_0+C)$, where $n\in \NN^*$, \\  
$C=T(G)\oplus \ZZ x_1\oplus \cdots \oplus \ZZ x_m \oplus G_0$, where $x_1,\dots, 
x_m$ are torsion-free elements of $G$ whose classes $\overline{x_1}, \cdots,\overline{x_m}$ 
modulo $G_0$ are rationally independent within $K(G)$,\\
the class $\overline{x_0}$ of $x_0$ modulo $G_0$ generates $T(K(G))$, if $n>1$ then 
there exists a generator 
$u$ of $T(G)$ such that $nx_0-u$ is a positive and cofinal element of $G_0$, and if 
$p$ divides $nx_0-u$ within $G$, then $p$ and $n$ are coprime.\\ 
Furthemore, the t.o.g. $G_0$ is equal to a lexicographically ordered direct sum of 
finitely generated Ar\-chi\-me\-de\-an subgroups. 
\end{prop}
\begin{proof}. 
By \cite{Sc 82} Lemma 1.2 p. 2, $G_0$ is equal to a lexicographically ordered direct sum of 
finitely generated Archimedean subgroups. 
Since $G$ is finitely generated, $K(G)=G/G_0$ is finitely generated, and by 
\cite{La 65} Theorem 4.8 p. 49, there exist $x_1,\dots,x_m$ in $G$ such that 
$K(G)=T(K(G))\oplus \ZZ \overline{x_1}\oplus \cdots \oplus \ZZ\overline{x_m}$ and 
$\overline{x_1}, \cdots,\overline{x_m}$ are rationally independent within $K(G)$. 
Furthermore, $T(K(G))$ is a finitely generated subgroup of $\KK$, hence it is cyclic and finite, 
let $x_0\in G$ be such that $T(K(G))$ is generated by $\overline{x_0}$. 
In the same way, $T(G)$ is cyclic and finite, we let $u$ be a generator of $T(G)$. Denote by $l$ 
the cardinal of $T(G)$ and by  
$n$ the lowest element of $\NN^*$ such that $n\overline{x_0}=\overline{u}$. 
If $n=1$, we can assume that $x_0=u$. If $n>1$, if necessary, we can take $x_0+y$ in place 
of $x_0$, 
where $y$ is a positive and cofinal element of $G_0$, and we  assume that $nx_0-u$ is positive and cofinal within $G_0$. \\ 
\indent Set $y=nx_0-u$, 
and let $d$ be a divisor of $n$ such that $d$ divides $nx_0-u$ within $G$, say $y=dy'$, 
and $n=dn'$. Then $u=d(n'x_0-y')$, and since $T(G)$ is pure in $G$ we have: 
$u'=n'x_0-y'\in T(G)$. It follows that the cardinal of the quotient set 
$T(K(G))/T(G)$ is at most equal to $n'$. Now, we know that this cardinal is equal to $n$, 
hence $n'=n$ and $d=1$. \\ 
\indent Set $H=\ZZ x_1\oplus \cdots \oplus \ZZ x_m$; we have: $H\cap G_0=\{0\}$, and 
since $H \oplus G_0$ is torsion-free we also have: $T(G)\cap (H\oplus G_0)=\{0\}$, 
consequently, $G$ contains the subgroup $T(G)\oplus H \oplus G_0$. 
Let $x\in G$, $\overline{x}$ decomposes in an unique way as 
$\overline{x}=k\overline{x_0}+a_1\overline{x_1}+\cdots+a_m\overline{x_m}$, where the 
$a_i$'s belong to $\ZZ$ and $0\leq k < ln$ (the cardinal of $T(K(G))$). Let  
$k=nq+r$, where $q\geq 0$ and $0\leq r< n$, then 
$\overline{x}=r\overline{x_0}+q\overline{u}+a_1\overline{x_1}+\cdots+a_m\overline{x_m}$, 
and $x$ decomposes as $x=rx_0+qu+a_1x_1+\cdots + a_mx_m+y$, where  
$y \in G_0$, this concludes the proof. 
\end{proof} 
\begin{prop}\label{prop4} Let $G$ be an Abelian c.o.g. and $G'$ be another Abelian c.o.g. 
such that 
$T(G)$ embeds in $T(G')$, $K(G)$ embeds $K(G')$ (in both cases as c.o.g.), 
and if $G_0$ is nontrivial then $G_0'$ has no greatest proper convex subgroup and for every 
$x'$ in the positive cone of $G_0'$, there exists 
$y'\in G_0'$ such that $x'<y'$ and the Archimedean component of $G_0'$ associated to $y$ 
is isomorphic to $\RR$. Then $G$ embeds in $G'$. 
\end{prop} 
\begin{proof}. 
Denote by $\varphi$ the embedding of $K(G)$ in $K(G')$, and let $x_0$, $x_1, 
\dots x_m$ and $u$ be defined in the same way as in Proposition \ref{prop3}. 
Let $x_1',\dots ,x_m'$ in $G'$ be such that $\varphi(\overline{x_1})=\overline{x_1'}, \dots, 
\varphi(\overline{x_m})=
\overline{x_m'}$; $x_1',\dots,x_m'$ are rationally independent because 
$\overline{x_1}, \dots,\overline{x_m}$ are rationally independent. 
Set $H'=\ZZ x_1'\oplus \cdots \oplus \ZZ x_m'$. 
Since $T(G')$ embeds in $T(K(G'))$, and $\KK$ contains one and only one subgroup of every cardinal, 
$T(G')$ contains an element $u'$ which has the same torsion as $u$ and such that 
$\overline{u'}=\varphi(\overline{u})$. Then $T(G)\oplus H=\langle u\rangle \oplus H$ 
is isomorphic to $\langle u' \rangle \oplus H'$. If $G_0=\{0\}$, then $G=T(G)\oplus H$ 
embeds in $G'$. Assume that $G_0\neq \{0\}$, and decompose it as a lexicographically ordered 
direct sum of finitely generated Archimedean subgroups 
$G_0=A_s\oplus A_{s-1} \oplus \cdots \oplus A_1$. 
Let $x'\in G'$ be such that $\overline{x'}=\varphi(\overline{x_0})$ and $nx'-u'>0$. 
According to the hypothesis, $G_0'$ contains elements $0<y_1'<\cdots<y_{s-1}'<y_s'$ 
such that $H_2(y_1')<\cdots <H_2(y_s')$, and for $1\leq j\leq s$, 
$H_2(y_j')/H_1(y_j')\simeq \RR$, and we can assume that $y_s'\geq nx'-u'$. 
If $H_2(nx'-u')<H_2(y'_s)$, we take $x'+y_s'$ in place of $x'$. Since 
$G_0$ is finitely generated, there exists a greatest integer $p$ which divides $y$ within $G_0$. 
According to Proposition \ref{prop3}, $p$ and $n$ are coprime. If $p$ divides $y'$ 
within $G_0'$, we set $x_0'=x'$. Otherwise, 
we let $\alpha$ and $\beta$ in $\ZZ$ be such that $\alpha p+\beta n=1$, and 
we set $x_0'=(1-\beta n)x'+\beta u'$. We let $\frac{1}{p}(nx_0'-u')$ be the image of 
$\frac{1}{p}(nx_0-u)$. 
If $G_0\not\simeq \ZZ$, there exists an unique embedding 
$f$ from $A_s$ into $\RR$ such that $f(\frac{1}{p}(nx_0-u))=1$, and  
an unique isomorphism $g$ from $H_2(nx_0'-u')/H_1(nx_0'-u')$ onto $\RR$ such that 
$g(\frac{1}{p}\overline{x_0}')=1$. 
We set $B_s'=g^{-1}\circ f(A_s)$, since $B_s'$ is finitely generated, there exists a 
subgroup $A_s'$ of $G_0'$ which contains $\frac{1}{p}(nx_0'-u')$ and such that 
$B_s'=A_s'/H_1(nx_0'-u')$. 
We define subgroups $A_{s-1}',\dots, A_1'$ in such a way that we get an 
ordered groups isomorphism between $G_0=A_s\oplus \cdots \oplus A_1$ 
and $A_s'\oplus \cdots \oplus A_1'$, which extends to an isomorphism of c.o.g. 
from $\langle u\rangle \oplus H \oplus G_0$ onto 
$\langle u'\rangle \oplus H' \oplus A_s'\oplus \cdots \oplus A_1'$. Finally, since the 
image of $nx_0-u$ is $nx_0'-u'$, this isomorphsim extends to an isomorphism 
from $G$ onto the subgroup of $G'$ which is generated by $x_0'$ and 
$\langle u'\rangle \oplus H' \oplus A_s'\oplus \cdots \oplus A_1'$. 
\end{proof}
\begin{cor}\label{cor5} 
Let $G$ and $G'$ be two Abelian c.o.g. having isomorphic torsion subgroups, 
where $G'$ is infinite and $\omega_1$-saturated, and such that either $K(G')$ is infinite, 
or $K(G')$ is isomorphic to $K(G)$. Then every finitely generated subgroup of $G$ 
embeds into $G'$. 
\end{cor} 
\begin{cor}\label{cor6} 
There are $2^{\aleph_0}$ distinct universal theories of Abelian c.o.g., each one determined 
by a couple of invariants which are subgroups of $\UU$: $T(G)$ and $K(G)$ if $K(G)$ is finite, 
$T(G)$ and $\UU$ if $K(G)$ is infinite. 
\end{cor}
\maketitle
\begin{tabbing}
giraudet@math\= a.dide \= Michele GIRAUDET.un espace\= abcde\= 
gerard.leloup@univ-le \kill
\>\>{\small{\rm Mich\a`ele GIRAUDET}} \>\> {\small{\rm G\a'erard LELOUP}}\\
\>{\small{\rm giraudet@math.univ-paris-diderot.fr}} \> 
\>{\small{\rm gerard.leloup@univ-lemans.fr}}
\end{tabbing}
\begin{center}{\small{\rm 
D\'epartement de Math\'ematiques\\
Facult\'e des Sciences\\
avenue Olivier Messiaen\\
72085 LE MANS CEDEX\\
FRANCE\\[2mm] 
Fran\c{c}ois LUCAS\\
LAREMA - UMR CNRS 6093\\
Universit\'e d'Angers\\Ê
2 boulevard Lavoisier\\
49045 ANGERS CEDEX 01\\ 
FRANCE\\
lucasfm49@gmail.com}}
\end{center}
\end{document}